\documentclass[12pt,reqno]{amsart}
\usepackage{amsmath,amssymb,amsthm,amscd,amsfonts,geometry,color}
\renewcommand{\labelenumi}{\rm(\theenumi)}

\newtheorem{theorem}{Theorem}[section]

\newtheorem{lemma}[theorem]{Lemma}
\newtheorem{proposition}[theorem]{Proposition}

\newtheorem{main}{Main Theorem}

\theoremstyle{definition}

\newtheorem{remark}[theorem]{Remark}

\newcommand{\0}{\mathbf{0}}                       

\newcommand{\fin}{\operatorname{Fin}}             
\newcommand{\doubl}{\operatorname{D}}             

\newcommand{\conti}{\operatorname{C}}             
\newcommand{\pseudo}{\operatorname{PM}}             
\newcommand{\metr}{\operatorname{M}}             
\newcommand{\adm}{\operatorname{AM}}             

\begin{document}

\title[Isometries between spaces of metrics]{Isometries between spaces of metrics}
\author{Katsuhisa Koshino}
\address[Katsuhisa Koshino]{Faculty of Engineering, Kanagawa University, 3-27-1 Rokkakubashi, Kanagawa-ku, Yokohama-shi, 221-8686, Japan}
\email{ft160229no@kanagawa-u.ac.jp}
\maketitle

\begin{abstract}
Given a metrizable space $Z$, denote by $\pseudo(Z)$ the space of continuous bounded pseudometrics on $Z$,
 and denote by $\adm(Z)$ the one of continuous bounded admissible metrics on $Z$,
 the both of which are equipped with the sup-norm $\|\cdot\|$.
Let ${\rm Pc}(Z)$ be the subspace of $\adm(Z)$ satisfying the following:
\begin{itemize}
 \item for every $d \in {\rm Pc}(Z)$, there exists a compact subset $K \subset Z$ such that if $d(x,y) = \|d\|$,
 then $x, y \in K$.
\end{itemize}
Moreover, set
 $${\rm Pp}(Z) = \{d \in \adm(Z) \mid \text{ there only exists } \{z,w\} \subset Z \text{ such that } d(z,w) = \|d\|\},$$
 and let $\metr(Z)$ be ${\rm Pc}(Z)$ or ${\rm Pp}(Z)$.
In this paper, we shall prove the Banach-Stone type theorem on spaces of metrics,
 that is, for metrizable spaces $X$ and $Y$, the following are equivalent:
 \begin{enumerate}
  \item $X$ and $Y$ are homeomorphic;
  \item there exists a surjective isometry $T : \pseudo(X) \to \pseudo(Y)$ with $T(\metr(X)) = \metr(Y)$;
  \item there exists a surjective isometry $T : \adm(X) \to \adm(Y)$ with $T(\metr(X)) = \metr(Y)$;
  \item there exists a surjective isometry $T : \metr(X) \to \metr(Y)$.
 \end{enumerate}
Then for each surjective isometry $T : \pseudo(X) \to \pseudo(Y)$ with $T(\metr(X)) = \metr(Y)$, there is a homeomorphism $\phi : Y \to X$ such that for any $d \in \pseudo(X)$ and for any $x, y \in Y$, $T(d)(x,y) = d(\phi(x),\phi(y))$.
Except for the case where the cardinality of $X$ or $Y$ is equal to $2$, the homeomorphism $\phi$ can be chosen uniquely.
\end{abstract}

\section{Introduction}

Isometries between function spaces have been studied in functional analysis.
The Banach-Stone theorem \cite{Banac,MHSto} is one of the most important results among those research,
 and its developments have been obtained until now, refer to \cite{FJ} as a historical note.
Throughout the paper, an isometry means a surjective isometry.
For a metrizable space $Z$, let $\conti(Z)$ be the space of continuous bounded real-valued functions on $Z$ with the sup-norm $\|\cdot\|$: for any $f \in \conti(Z)$, $\|f\| = \sup\{|f(z)| \mid z \in Z\}$.
Denote the positive cone by $\conti_+(Z) \subset \conti(Z)$.
Recently, L.~Sun, Y.~Sun and D.~Dai \cite{SSD} showed the Banach-Stone type theorem on positive cones of continuous function spaces as follows:

\begin{theorem}
Suppose that $X$ and $Y$ are compact metrizable spaces.
Then $X$ and $Y$ are homeomorphic if and only if $\conti_+(X)$ and $\conti_+(Y)$ are isometric.
\end{theorem}

D.~Hirota, I.~Matsuzaki and T.~Miura \cite{HMM} generalized the above theorem in the non-compact case.
In this paper, we shall establish the Banach-Stone type theorem on spaces of metrics.
Let $\pseudo(Z) \subset \conti_+(Z^2)$ be the subspace consisting of continuous bounded pseudometrics on $Z$,
 and let $\adm(Z) \subset \pseudo(Z)$ be the subspace consisting of continuous bounded admissible metrics.
As is easily observed,
 $\pseudo(X)$ and $\pseudo(Y)$ (respectively, $\adm(X)$ and $\adm(Y)$) are isometric if metrizable spaces $X$ and $Y$ are homeomorphic.
When $X$ and $Y$ are compact,
 M.E.~Shanks \cite{Sh} showed the converse of it and established the following:

\begin{theorem}
Let $X$ and $Y$ be compact metrizable spaces.
The following are equivalent:
\begin{enumerate}
 \item $X$ and $Y$ are homeomorphic;
 \item $\pseudo(X)$ and $\pseudo(Y)$ are isometric;
 \item $\adm(X)$ and $\adm(Y)$ are isometric.
\end{enumerate}
\end{theorem}

Shanks focused on certain lattice structure on equivalent classes of $\pseudo(X)$,
 which determined the topology of $X$,
 and the method was different from those of S.~Banach and M.H.~Stone.
However, Shanks did not give descriptions of isometries by using homeomorphisms like the canonical formula ($\ast$) as in Main Theorem,
 which appeared in the Banach-Stone theorem.
Set ${\rm Pc}(Z)$ be the subspace of $\adm(Z)$ that satisfies the following condition:
\begin{itemize}
 \item for every $d \in {\rm Pc}(Z)$, there is a compact set $K$ in $Z$ such that if $d(x,y) = \|d\|$,
 then $x, y \in K$.
\end{itemize}
Notice that $\adm(Z) = {\rm Pc}(Z)$ when $Z$ is compact.
Moreover, put
 $${\rm Pp}(Z) = \{d \in \adm(Z) \mid \text{ there only exists } \{z,w\} \subset Z \text{ such that } d(z,w) = \|d\|\},$$
 and let $\metr(Z)$ be ${\rm Pc}(Z)$ or ${\rm Pp}(Z)$.
We shall generalize this result and determine isometries by homeomorphisms as follows:

\begin{main}\label{isometry}
Let $X$ and $Y$ be metrizable spaces.
The following are equivalent:
\begin{enumerate}
 \item $X$ and $Y$ are homeomorphic;
 \item there exists an isometry $T : \pseudo(X) \to \pseudo(Y)$ with $T(\metr(X)) = \metr(Y)$;
 \item there exists an isometry $T : \adm(X) \to \adm(Y)$ with $T(\metr(X)) = \metr(Y)$;
 \item there exists an isometry $T : \metr(X) \to \metr(Y)$.
\end{enumerate}
In this case, for each isometry $T : \pseudo(X) \to \pseudo(Y)$ with $T(\metr(X)) = \metr(Y)$, there exists a homeomorphism $\phi : Y \to X$ such that for any $d \in \pseudo(X)$ and for any $x, y \in Y$,
\begin{equation}
 T(d)(x,y) = d(\phi(x),\phi(y)). \tag{$\ast$}
\end{equation}
Except for the case where the cardinality of $X$ or $Y$ is equal to $2$, the homeomorphism $\phi$ can be chosen uniquely.
\end{main}

\section{Spaces of metrics}

In this section, we shall review the study on spaces of metrics.
Recently, Y.~Ishiki \cite{Ish1,Ish2,Ish7} have researched topologies of spaces of metrics.
The author \cite{Kos20,Kos26} investigated their Borel hierarchy, complete metrizability and topological types,
 and proved the following:

\begin{theorem}
Let $\kappa$ be a cardinal and $\ell_2(\kappa)$ be the Hilbert space of density $\kappa$.
When a metrizable space $Z$ is of density $\kappa$,
 the space $\pseudo(Z)$ is homeomorphic to
 \begin{enumerate}
 \renewcommand{\labelenumi}{(\roman{enumi})}
  \item $[0,1)^{\kappa(\kappa - 1)/2}$ if $Z$ is finite;
  \item $\ell_2(2^{< \kappa})$ if $Z$ is infinite and generalized compact;
  \item $\ell_2(2^\kappa)$ if $Z$ is not generalized compact.
 \end{enumerate}
Additionally, when $Z$ is infinite and $\sigma$-compact,
 the subspace $\adm(Z)$ is homeomorphic to
 \begin{enumerate}
 \renewcommand{\labelenumi}{(\roman{enumi})}
  \item $\ell_2(\aleph_0)$ if $Z$ is compact;
  \item $\ell_2(2^{\aleph_0})$ if $Z$ is not compact.
 \end{enumerate}
\end{theorem}

This means that for metrizable spaces $X$ and $Y$, even if $\pseudo(X)$ and $\pseudo(Y)$ (respectively, $\adm(X)$ and $\adm(Y)$) are homeomorphic,
 $X$ and $Y$ are not necessarily homeomorphic.
On metric structures of spaces of metrics, Y.~Ishiki and the author \cite{IK} studied their isometric universality.

As a basic property on metrics, we have the following, see \cite[Lemma~2.1]{IK}.

\begin{lemma}\label{add.}
For a metrizable space $Z$, for every $d \in \pseudo(Z)$ and every $\rho \in \adm(Z)$, their sum $d + \rho \in \adm(Z)$.
\end{lemma}

F.~Hausdorff \cite{Hau} showed the metric extension theorem,
 which states that for every metrizable space $Z$ and its closed subset $A \subset Z$, any $d \in \adm(X)$ can be extended over $X$.
We may obtain the pseudometric version of it,
 that is preserving their norms, as follows, refer to \cite{Ish8} for example.

\begin{theorem}\label{ext.}
Suppose that $Z$ is a metrizable space and $A \subset Z$ is a closed subset.
For each $d \in \pseudo(A)$, there is $\tilde{d} \in \pseudo(Z)$ such that $\tilde{d}|_{A^2} = d$ and $\|\tilde{d}\| = \|d\|$.
\end{theorem}

This result will play a key role in the present paper instead of Urysohn's lemma,
 which have been frequently used in proving the various Banach-Stone type theorems.

For a metrizable space $Z$ with a pseudometric $d \in \pseudo(Z)$, denote the open ball centered at $z \in Z$ of radius $r > 0$ by
 $$B_d(z,r) = \{w \in Z \mid d(z,w) < r\}$$
 and the closed ball by
 $$\overline{B}_d(z,r) = \{w \in Z \mid d(z,w) \leq r\}.$$
Note that since $d$ is continuous,
 $B_d(z,r)$ is open and $\overline{B}_d(z,r)$ is closed in $Z$.
Recall that the sup-metric induced by $\|\cdot\|$ is complete on $\conti(Z^2)$ and $\pseudo(Z)$,
 which is closed in $\conti(Z^2)$.
Since $\adm(Z)$ is dense in $\pseudo(Z)$, see \cite[Proposition~5]{Kos20},
 an isometry $T : \adm(X) \to \adm(Y)$ can be extended to an isometry $\tilde{T} : \pseudo(X) \to \pseudo(Y)$.
Furthermore, we will prove that ${\rm Pc}(Z)$ is also a dense subset of $\pseudo(Z)$.

\begin{proposition}\label{dense}
For every metrizable space $Z$, the subset ${\rm Pp}(Z)$ is dense in $\pseudo(Z)$,
 and hence so is ${\rm Pc}(Z)$.
\end{proposition}

\begin{proof}
The cases where $Z = \emptyset$ and where $Z$ is a singleton are trivial,
 so suppose that $Z$ is non-degenerate.
Fix any $d \in \pseudo(Z)$ and any $\epsilon > 0$.
Since $\adm(Z)$ is dense in $\pseudo(X)$,
 we may assume that $d \in \adm(Z)$.
Note that $d$ is bounded,
 so choose distinct points $x, y \in Z$ so that $d(x,y) \geq \|d\| - \epsilon$.
Let $a = d(x,y)$,
 where we may also assume that $a \geq \epsilon$ by replacing $\epsilon$ with a sufficient small positive number.
Applying Theorem~\ref{ext.}, we can find a pseudometric $\rho_n \in \pseudo(Z)$ for every natural number $n \geq 1$ such that
\begin{enumerate}
\renewcommand{\theenumi}{\roman{enumi}}
 \item $\rho(z,w) = 4\epsilon$ if $z \in \overline{B}_d(x,\epsilon/2^{n + 1})$ and $w \in \overline{B}_d(y,\epsilon/2^{n + 1})$;
 \item $\rho(z,w) = 2\epsilon$ if $z \in Z \setminus (B_d(x,\epsilon/2^n) \cup B_d(y,\epsilon/2^n))$ and $w \in \overline{B}_d(x,\epsilon/2^{n + 1}) \cup \overline{B}_d(y,\epsilon/2^{n + 1})$;
 \item $\rho(z,w) = 0$ if $z, w \in Z \setminus (B_d(x,\epsilon/2^n) \cup B_d(y,\epsilon/2^n))$,
 if $z, w \in \overline{B}_d(x,\epsilon/2^{n + 1})$,
 or if $z, w \in \overline{B}_d(y,\epsilon/2^{n + 1})$;
 \item $\rho(z,w) \leq 4\epsilon$ if otherwise.
\end{enumerate}
Define $\rho = \sum_{n = 1}^\infty \rho_n/2^n$,
 so $\rho \in \pseudo(Z)$ because $\pseudo(Z)$ is complete.
Observe that
 $$\|(d + \rho) - d\| = \|\rho\| \leq \sum_{n = 1}^\infty \|\rho_n\| \leq \sum_{n = 1}^\infty 4\epsilon/2^n = 4\epsilon.$$
It is only needed to verify that $d + \rho \in {\rm Pp}(Z)$.
Remark that $d + \rho \in \adm(Z)$ by Lemma~\ref{add.}.
Define the closed subsets
 $$Z_1 = Z \setminus (B_d(x,\epsilon/2) \cup B_d(y,\epsilon/2)),$$
 and for every $n \geq 2$,
 $$Z_n = (\overline{B}_d(x,\epsilon/2^{n - 1}) \cup \overline{B}_d(y,\epsilon/2^{n - 1})) \setminus (B_d(x,\epsilon/2^n) \cup B_d(y,\epsilon/2^n)).$$
\begin{enumerate}
 \item When $z = x$ or $z = y$,
 and $w \in Z_n$,
 \begin{align*}
  d(z,w) + \rho(z,w) &= d(z,w) + \sum_{n = 1}^\infty \rho_n(z,w)/2^n \leq a + \epsilon/2^{n - 1} + \sum_{i = 1}^{n - 1} 4\epsilon/2^i + \sum_{i = n}^\infty 2\epsilon/2^i\\
  &\leq a + 4\epsilon/2^{n + 1} + \sum_{i = 1}^{n - 1} 4\epsilon/2^i + \sum_{i = n + 1}^\infty 4\epsilon/2^i < a + \sum_{i = 1}^\infty 4\epsilon/2^i = a + 4\epsilon.
 \end{align*}
 \item When $z, w \in Z_n$,
 \begin{align*}
  d(z,w) + \rho(z,w) &= d(z,w) + \sum_{n = 1}^\infty \rho_n(z,w)/2^n \leq a + \epsilon/2^{n - 2} + \sum_{i = 1}^{n - 1} 4\epsilon/2^i\\
  &\leq a + 4\epsilon/2^n + \sum_{i = 1}^{n - 1} 4\epsilon/2^i < a + \sum_{i = 1}^\infty 4\epsilon/2^i = a + 4\epsilon.
 \end{align*}
 \item When $z \in Z_n$ and $w \in Z_{n + 1}$,
 \begin{align*}
  d(z,w) + \rho(z,w) &= d(z,w) + \sum_{n = 1}^\infty \rho_n(z,w)/2^n \leq a + \epsilon/2^{n - 1} + \epsilon/2^n + \sum_{i = 1}^n 4\epsilon/2^i\\
  &\leq a + 4\epsilon/2^{n + 1} + 4\epsilon/2^{n + 2} + \sum_{i = 1}^n 4\epsilon/2^i < a + \sum_{i = 1}^\infty 4\epsilon/2^i = a + 4\epsilon.
 \end{align*}
 \item When $z \in Z_n$ and $w \in Z_m$, $m \geq n + 2$,
 \begin{align*}
  d(z,w) + \rho(z,w) &= d(z,w) + \sum_{n = 1}^\infty \rho_n(z,w)/2^n\\
  &\leq a + \epsilon/2^{n - 1} + \epsilon/2^{m - 1} + \sum_{i = 1}^{n - 1} 4\epsilon/2^i + \sum_{i = n}^{m - 2} 2\epsilon/2^i + 4\epsilon/2^{m - 1}\\
  &\leq a + 4\epsilon/2^{n + 1} + 4\epsilon/2^{m + 1} + \sum_{i = 1}^{n - 1} 4\epsilon/2^i + \sum_{i = n + 1}^{m - 1} 4\epsilon/2^i + 4\epsilon/2^{m - 1}\\
  &= a + 4\epsilon/2^n - 4\epsilon/2^{n + 1} + 4\epsilon/2^{m + 1} + \sum_{i = 1}^{n - 1} 4\epsilon/2^i + \sum_{i = n + 1}^{m - 1} 4\epsilon/2^i + 4\epsilon/2^{m - 1}\\
  &< a + \sum_{i = 1}^\infty 4\epsilon/2^i = a + 4\epsilon.
 \end{align*}
\end{enumerate}
To sum up, we have that $d(z,w) + \rho(z,w) < a + 4\epsilon$ for all pairs $(z,w) \in Z^2 \setminus \{(x,y),(y,x)\}$.
By the definition of $\rho$,
 $$d(x,y) + \rho(x,y) = d(x,y) + \sum_{n = 1}^\infty \rho_n(x,y)/2^n = a + \sum_{n = 1}^\infty 4\epsilon/2^n = a + 4\epsilon.$$
The proof is completed.
\end{proof}

\section{The peaking function argument}

We shall prove Main Theorem by using the peaking function argument,
 which is based on \cite{HMM}, and traces its history back to Stone's method,
 that is, our strategy is different from Shanks' one.
From now on, let $X$ and $Y$ be non-degenerate metrizable spaces and $T : \pseudo(X) \to \pseudo(Y)$ be an isometry.
Assume that $\0$ is the zero function,
 that is a pseudometric.
By the same argument as Lemmas~2.3 and 2.4 of \cite{HMM}, we have the following:

\begin{lemma}\label{0}
For every $d \in \pseudo(X)$, if
 $$\max\{\|T(\0)\|,\|T^{-1}(\0)\|\} < \|d\|,$$
 then $\|T(d)\| = \|d\|$.
Similarly, for every $\rho \in \pseudo(Y)$, if
 $$\max\{\|T(\0)\|,\|T^{-1}(\0)\|\} < \|\rho\|,$$
 then $\|T^{-1}(\rho)\| = \|\rho\|$.
\end{lemma}

It follows from the above lemma that the isometry $T$ is norm-preserving.

\begin{proposition}\label{norm}
The equalities $T(\0) = \0$ and $T^{-1}(\0) = \0$ hold.
Hence $\|T(d)\| = \|d\|$ for any $d \in \pseudo(X)$ and $\|T^{-1}(\rho)\| = \|\rho\|$ for any $\rho \in \pseudo(Y)$.
\end{proposition}

\begin{proof}
Assume that $T(\0) \neq \0$,
 so there exist points $x, y \in Y$ such that $T(\0)(x,y) > 0$.
Then
 \begin{align*}
  \max\{\|T(\0)\|,\|T^{-1}(\0)\|\} &< 2\|T(\0)\| + \|T^{-1}(\0)\| = \frac{2\|T(\0)\| + \|T^{-1}(\0)\|}{\|T(\0)\|} \|T(\0)\|\\
  &= \bigg\|\frac{2\|T(\0)\| + \|T^{-1}(\0)\|}{\|T(\0)\|} T(\0)\bigg\|.
 \end{align*}
Due to Lemma~\ref{0},
 $$\bigg\|T^{-1}\bigg(\frac{2\|T(\0)\| + \|T^{-1}(\0)\|}{\|T(\0)\|} T(\0)\bigg)\bigg\| = \bigg\|\frac{2\|T(\0)\| + \|T^{-1}(\0)\|}{\|T(\0)\|} T(\0)\bigg\|.$$
Since $T$ is an isometry,
 \begin{align*}
  2\|T(\0)\| + \|T^{-1}(\0)\| &> 2\|T(\0)\| + \|T^{-1}(\0)\| - \|T(\0)\| = \bigg(\frac{2\|T(\0)\| + \|T^{-1}(\0)\|}{\|T(\0)\|} - 1\bigg)\|T(\0)\|\\
  &= \bigg\|\bigg(\frac{2\|T(\0)\| + \|T^{-1}(\0)\|}{\|T(\0)\|} - 1\bigg)T(\0)\bigg\|\\
  &= \bigg\|\frac{2\|T(\0)\| + \|T^{-1}(\0)\|}{\|T(\0)\|} T(\0) - T(\0)\bigg\|\\
  &= \bigg\|T^{-1}\bigg(\frac{2\|T(\0)\| + \|T^{-1}(\0)\|}{\|T(\0)\|} T(\0)\bigg) - T^{-1}(T(\0))\bigg\|\\
  &= \bigg\|T^{-1}\bigg(\frac{2\|T(\0)\| + \|T^{-1}(\0)\|}{\|T(\0)\|} T(\0)\bigg) - \0\bigg\|\\
  &= \bigg\|T^{-1}\bigg(\frac{2\|T(\0)\| + \|T^{-1}(\0)\|}{\|T(\0)\|} T(\0)\bigg)\bigg\| = \bigg\|\frac{2\|T(\0)\| + \|T^{-1}(\0)\|}{\|T(\0)\|} T(\0)\bigg\|\\
  &= 2\|T(\0)\| + \|T^{-1}(\0)\|.
 \end{align*}
This is a contradiction.
We conclude that $T(\0) = \0$.
Moreover, for each $d \in \pseudo(X)$,
 $$\|T(d)\| = \|T(d) - \0\| = \|T(d) - T(\0)\| = \|d - \0\| = \|d\|.$$
Similarly, $T^{-1}(\0) = \0$ and $\|T^{-1}(\rho)\| = \|\rho\|$ for any $\rho \in \pseudo(Y)$.
\end{proof}

For a metrizable space $Z$, let $\fin_2(Z)$ be the hyperspace consisting of singletons and doubletons of $Z$ endowed with the Vietoris topology,
 and let
 $$\doubl(Z) = \{\{x,y\} \in \fin_2(Z) \mid x \neq y\}.$$
For each $\{x,y\} \in \fin_2(Z)$, put
 $$\mathcal{P}(Z,\{x,y\}) = \{d \in \pseudo(Z) \mid d(x,y) = \|d\|\}.$$
Given a pseudometric $d \in \pseudo(Z)$, we define
 $$\mathcal{F}(Z,d) = \{\{x,y\} \in \fin_2(Z) \mid d(x,y) = \|d\|\}.$$

\begin{lemma}\label{sum}
Fix any $\{x,y\} \in \fin_2(X)$ and any $d_i \in \mathcal{P}(X,\{x,y\})$, $1 \leq i \leq n$.
Then the sum $d = \sum_{i = 1}^n d_i \in \mathcal{P}(X,\{x,y\})$.
Furthermore, if $T({\rm Pc}(X)) \subset {\rm Pc}(Y)$ and each $d_i \in {\rm Pc}(X)$,
 there exists $\{z,w\} \in \fin_2(Y)$ such that $T(d)(z,w) = d(x,y)$.
\end{lemma}

\begin{proof}
Observe that
 $$d(x,y) \leq \|d\| = \Bigg\|\sum_{i = 1}^n d_i\Bigg\| \leq \sum_{i = 1}^n \|d_i\| = \sum_{i = 1}^n d_i(x,y) = d(x,y),$$
 so $d(x,y) = \|d\|$.
Therefore $d \in \mathcal{P}(X,\{x,y\})$.
To show that the latter part, take a compact subset $K_i \subset X$, $1 \leq i \leq n$, such that for any $u, v \in X$ with $d_i(u,v) = \|d_i\|$, $u, v \in K_i$.
If $d(u,v) = \|d\|$,
 then
 $$\sum_{i = 1}^n d_i(u,v) = d(u,v) = \|d\| = d(x,y) = \sum_{i = 1}^n d_i(x,y).$$
Since for each $i \in \{1, \ldots, n\}$, $d_i(u,v) \leq \|d_i\| = d_i(x,y)$,
 we can get that $d_i(u,v) = d_i(x,y) = \|d_i\|$.
Hence the points $u$ and $v$ are contained in any $K_i$,
 which implies that $d \in {\rm Pc}(X)$.
By the assumption of $T$, $T(d) \in {\rm Pc}(Y)$,
 that is, there is a compact set $L \subset Y$ such that for any $z, w \in Y$, if $T(d)(z,w) = \|T(d)\|$,
 then $z, w \in L$.
Since $L$ is compact and $T(d)$ is continuous,
 there is $\{z,w\} \in \mathcal{F}(Y,T(d))$.
Then according to Proposition~\ref{norm},
 $$T(d)(z,w) = \|T(d)\| = \|d\| = d(x,y).$$
The proof is finished.
\end{proof}

Moreover, we have the following:

\begin{lemma}\label{f.i.p.}
Suppose that $T({\rm Pc}(X)) \subset {\rm Pc}(Y)$.
For any $\{x,y\} \in \fin_2(X)$ and any $d_i \in \mathcal{P}(X,\{x,y\}) \cap {\rm Pc}(X)$, $1 \leq i \leq n$, the intersection $\bigcap_{i = 1}^n \mathcal{F}(Y,T(d_i)) \neq \emptyset$.
\end{lemma}

\begin{proof}
Let $d = \sum_{i = 1}^n d_i$ and take a doubleton $\{z,w\} \in \fin_2(Y)$ such that $T(d)(z,w) = d(x,y)$ as in Lemma~\ref{sum}.
For each $i \in \{1, \ldots, n\}$, set $\rho_i = d - d_i$,
 so $\rho_i \in \mathcal{P}(X,\{x,y\})$ and $\rho_i(x,y) = \|\rho_i\|$ by Lemma~\ref{sum}.
Since $T$ is isometric,
 we have that
 $$T(d)(z,w) - T(d_i)(z,w) \leq \|T(d) - T(d_i)\| = \|d - d_i\| = \|\rho_i\| = \rho_i(x,y),$$
 and that according to Proposition~\ref{norm},
 $$\|T(d_i)\| = \|d_i\| = d_i(x,y) = d(x,y) - \rho_i(x,y) = T(d)(z,w) - \rho_i(x,y) \leq T(d_i)(z,w) \leq \|T(d_i)\|.$$
Thus $T(d_i)(z,w) = \|T(d_i)\|$,
 which implies that $\{z,w\} \in \mathcal{F}(Y,T(d_i))$.
Consequently, the intersection $\bigcap_{i = 1}^n \mathcal{F}(Y,T(d_i))$ is not empty.
\end{proof}

Due to the similar method to Proposition~\ref{dense}, we can prove the following:

\begin{lemma}\label{transl.}
Let $Z$ be a non-degenerate metrizable space.
For each $d \in \pseudo(Z)$ and each $\{x,y\} \in \doubl(Z)$, there is $\rho \in \mathcal{P}(Z,\{x,y\})$ such that $d + \rho \in \mathcal{P}(Z,\{x,y\}) \cap {\rm Pp}(Z)$.
\end{lemma}

\begin{proof}
Adding an admissible metric in $\mathcal{P}(X,\{x,y\})$, we may assume that $d \in \adm(Z)$ according to Lemma~\ref{add.}.
Let $a = d(x,y) > 0$ and
 $$b = \min\Big\{\max_{z \in Z} d(x,z),\max_{z \in Z} d(y,z)\Big\}.$$
Remark that $a \leq b$.
Using Theorem~\ref{ext.}, we can obtain $\rho_n \in \pseudo(Z)$ for each natural number $n \geq 1$ such that
\begin{enumerate}
\renewcommand{\theenumi}{\roman{enumi}}
 \item $\rho(z,w) = 4b$ if $z \in \overline{B}_d(x,a/2^{n + 1})$ and $w \in \overline{B}_d(y,a/2^{n + 1})$;
 \item $\rho(z,w) = 2b$ if $z \in Z \setminus (B_d(x,a/2^n) \cup B_d(y,a/2^n))$ and $w \in \overline{B}_d(x,a/2^{n + 1}) \cup \overline{B}_d(y,a/2^{n + 1})$;
 \item $\rho(z,w) = 0$ if $z, w \in Z \setminus (B_d(x,a/2^n) \cup B_d(y,a/2^n))$,
 if $z, w \in \overline{B}_d(x,a/2^{n + 1})$,
 or if $z, w \in \overline{B}_d(y,a/2^{n + 1})$;
 \item $\rho(z,w) \leq 4b$ if otherwise.
\end{enumerate}
By the same argument as Proposition~\ref{dense}, $\rho = \sum_{n = 1}^\infty \rho_n/2^n$ is the desired pseudometric.
Indeed, $\rho(x,y) = 4b = \|\rho\|$,
 and hence $\rho \in \mathcal{P}(Z,\{x,y\})$.
Moreover, for every pair $(z,w) \in Z^2 \setminus \{(x,y),(y,x)\}$,
 $$d(x,y) + \rho(x,y) = a + 4b = \|d + \rho\| > d(z,w) + \rho(z,w),$$
 so $d + \rho \in \mathcal{P}(Z,\{x,y\}) \cap {\rm Pp}(Z)$.
We complete the proof.
\end{proof}

Using the finite intersection property in compact spaces, see \cite[Theorem~3.1.1]{E},
 we can obtain the following lemma.

\begin{lemma}\label{intersect.}
Assume that $T({\rm Pc}(X)) \subset {\rm Pc}(Y)$.
For every doubleton $\{x,y\} \in \doubl(X)$,
 $$\bigcap_{d \in \mathcal{P}(X,\{x,y\}) \cap {\rm Pc}(X)} \mathcal{F}(Y,T(d)) \neq \emptyset.$$
\end{lemma}

\begin{proof}
Remark that $\mathcal{P}(X,\{x,y\}) \cap {\rm Pc}(X) \neq \emptyset$ by virtue of Lemma~\ref{transl.}.
Fix any $d_0 \in \mathcal{P}(X,\{x,y\}) \cap {\rm Pc}(X)$,
 so $T(d_0) \in {\rm Pc}(Y)$ and there exists a compact subset $L$ of $Y$ such that for any $z, w \in Y$ with $T(d_0)(z,w) = \|T(d_0)\|$, $z, w \in L$.
Since $L$ is compact and
 $$i : L^2 \ni (z,w) \mapsto \{z,w\} \in \fin_2(L)$$
 is surjective and continuous due to \cite[Lemma~5.3.4]{vM2},
 $\fin_2(L)$ is also compact.
Observe that $\mathcal{F}(Y,T(d))$ is closed in $\fin_2(Y)$ for every $d \in \mathcal{P}(X,\{x,y\})$.
Indeed, fix any $\{z,w\} \in \fin_2(Y) \setminus \mathcal{F}(Y,T(d))$.
Then $T(d)(z,w) < \|T(d)\|$.
Since $T(d)$ is continuous,
 we can find open neighborhoods $U$ of $z$ and $V$ of $w$ such that if $u \in U$ and $v \in V$,
 then $T(d)(u,v) < \|T(d)\|$.
Remark that the subset
 $$\mathcal{U} = \{\{u,v\} \in \fin_2(Y) \mid \{u,v\} \cap U \neq \emptyset, \{u,v\} \cap V \neq \emptyset, \text{ and } \{u,v\} \subset U \cup V\}$$
 is an open neighborhood of $\{z,w\}$ in $\fin_2(Y)$.
If $\{u,v\} \in \mathcal{U}$,
 then we may assume that $u \in U$ and $v \in V$,
 and hence $T(d)(u,v) < \|T(d)\|$.
Therefore $\{u,v\} \in \fin_2(Y) \setminus \mathcal{F}(Y,T(d))$,
 which means that $\mathcal{F}(Y,T(d))$ is closed in $\fin_2(Y)$.
Thus the set $\mathcal{F}(Y,T(d_0)) \subset \fin_2(L)$ is compact.
According to Lemma~\ref{f.i.p.}, the family
 $$\{\mathcal{F}(Y,T(d)) \cap \mathcal{F}(Y,T(d_0)) \mid d \in \mathcal{P}(X,\{x,y\})\}$$
 has the finite intersection property,
 and hence
 $$\bigcap_{d \in \mathcal{P}(X,\{x,y\}) \cap {\rm Pc}(X)} \mathcal{F}(Y,T(d)) = \bigcap_{d \in \mathcal{P}(X,\{x,y\}) \cap {\rm Pc}(X)} (\mathcal{F}(Y,T(d)) \cap \mathcal{F}(Y,T(d_0))) \neq \emptyset.$$
We finish the proof.
\end{proof}

We have certain uniqueness of peaks of metrics as follows:

\begin{lemma}\label{peak}
Suppose that $Z$ is a non-degenerate metrizable space.
For all doubletons $\{x,y\}, \{x',y'\} \in \doubl(Z)$, if
 $$\mathcal{P}(Z,\{x,y\}) \cap {\rm Pp}(Z) \subset \mathcal{P}(Z,\{x',y'\}),$$
 then $\{x,y\} = \{x',y'\}$.
\end{lemma}

\begin{proof}
Suppose not,
 so we may assume that $x \notin \{x',y'\}$ or $x' \notin \{x,y\}$.
By Theorem~\ref{ext.}, we can obtain $d \in \pseudo(X)$ such that $d(x,y) = 1$, $d(x',y') = 0$, and $\|d\| = 1$.
Moreover, there exists an admissible metric $\rho \in \pseudo(X)$ such that $\rho(z,w) < \rho(x,y) = \|\rho\|$ for any $(z,w) \in Z^2 \setminus \{(x,y),(y,x)\}$ by the same argument as Lemma~\ref{transl.}.
Then due to Lemma~\ref{add.},
 $$d + \rho \in (\mathcal{P}(Z,\{x,y\}) \cap {\rm Pp}(Z)) \setminus \mathcal{P}(Z,\{x',y'\}),$$
 which is a contradiction.
Hence $\{x,y\} = \{x',y'\}$.
\end{proof}

Now the following corresponding between $\doubl(X)$ and $\doubl(Y)$ will be given.

\begin{proposition}\label{bij.}
Suppose that $T({\rm Pc}(X)) = {\rm Pc}(Y)$.
There exists a bijection $\Phi : \doubl(Y) \to \doubl(X)$ such that
 $$T(\mathcal{P}(X,\{x,y\}) \cap {\rm Pc}(X)) = \mathcal{P}(Y,\Phi^{-1}(\{x,y\})) \cap {\rm Pc}(Y)$$
 for every $\{x,y\} \in \doubl(X)$ and
 $$T^{-1}(\mathcal{P}(Y,\{z,w\}) \cap {\rm Pc}(Y)) = \mathcal{P}(X,\Phi(\{z,w\})) \cap {\rm Pc}(X)$$
 for every $\{z,w\} \in \doubl(Y)$.
\end{proposition}

\begin{proof}
For each $\{z,w\} \in \doubl(Y)$, by virtue of Lemma~\ref{intersect.}, we can choose $\{x,y\} \in \fin_2(X)$ so that
 $$\{x,y\} \in \bigcap_{\rho \in \mathcal{P}(Y,\{z,w\}) \cap {\rm Pc}(Y)} \mathcal{F}(X,T^{-1}(\rho)).$$
Note that $T^{-1}(\rho)(x,y) = \|T^{-1}(\rho)\|$ for all $\rho \in \mathcal{P}(Y,\{z,w\}) \cap {\rm Pc}(Y)$,
 so
 $$T^{-1}(\mathcal{P}(Y,\{z,w\}) \cap {\rm Pc}(Y)) \subset \mathcal{P}(X,\{x,y\}) \cap {\rm Pc}(X).$$
Furthermore, due to Proposition~\ref{norm},
 $$T^{-1}(\rho)(x,y) = \|T^{-1}(\rho)\| = \|\rho\| > 0,$$
 and hence $x \neq y$.
Similarly, there exists $\{u,v\} \in \doubl(Y)$ such that
 $$T(\mathcal{P}(X,\{x,y\}) \cap {\rm Pc}(X)) \subset \mathcal{P}(Y,\{u,v\}) \cap {\rm Pc}(Y).$$
Then we get that
\begin{multline*}
 \mathcal{P}(Y,\{z,w\}) \cap {\rm Pc}(Y) = T(T^{-1}(\mathcal{P}(Y,\{z,w\}) \cap {\rm Pc}(Y)))\\
 \subset T(\mathcal{P}(X,\{x,y\}) \cap {\rm Pc}(X)) \subset \mathcal{P}(Y,\{u,v\}) \cap {\rm Pc}(Y),
\end{multline*}
 and hence $\{z,w\} = \{u,v\}$ due to Lemma~\ref{peak}.
Therefore
 $$T^{-1}(\mathcal{P}(Y,\{z,w\}) \cap {\rm Pc}(Y)) = T^{-1}(T(\mathcal{P}(X,\{x,y\}) \cap {\rm Pc}(X))) = \mathcal{P}(X,\{x,y\}) \cap {\rm Pc}(X).$$
Assume that
 $$\{x',y'\} = \bigcap_{\rho \in \mathcal{P}(Y,\{z,w\}) \cap {\rm Pc}(Y)} \mathcal{F}(X,T^{-1}(\rho)),$$
 so we get that
 $$\mathcal{P}(X,\{x,y\}) \cap {\rm Pc}(X) = T^{-1}(\mathcal{P}(Y,\{z,w\}) \cap {\rm Pc}(Y)) = \mathcal{P}(X,\{x',y'\}) \cap {\rm Pc}(X).$$
Using Lemma~\ref{peak} again,
 we have that $\{x,y\} = \{x',y'\}$.
Hence we can define a map $\Phi : \doubl(Y) \to \doubl(X)$ by $\Phi(\{z,w\}) = \{x,y\}$.

On the other hand, fix any $\{x,y\} \in \doubl(X)$.
By the same argument as the above, we can uniquely choose $\{z,w\} \in \doubl(Y)$ so that
 $$T(\mathcal{P}(X,\{x,y\}) \cap {\rm Pc}(X)) = \mathcal{P}(Y,\{z,w\}) \cap {\rm Pc}(Y).$$
Then observe that
\begin{multline*}
 \mathcal{P}(X,\{x,y\}) \cap {\rm Pc}(X) = T^{-1}(T(\mathcal{P}(X,\{x,y\}) \cap {\rm Pc}(X)))\\
 = T^{-1}(\mathcal{P}(Y,\{z,w\}) \cap {\rm Pc}(Y)) = \mathcal{P}(X,\Phi(\{z,w\})) \cap {\rm Pc}(X),
\end{multline*}
 which implies that $\{x,y\} = \Phi(\{z,w\})$ by  Lemma~\ref{peak}.
Consequently, $\Phi$ is a bijection.
The proof is completed.
\end{proof}

\begin{remark}\label{card.}
Due to the same argument, we can also obtain a bijection $\Phi : \doubl(Y) \to \doubl(X)$ as in the above proposition under the assumption $T({\rm Pp}(X)) = {\rm Pp}(Y)$.
If $T(\metr(X)) = \metr(Y)$,
 then the cardinality of $X$ is coincident with the one of $Y$.
\end{remark}

\section{Constructing a bijection between $X$ and $Y$}

From now on, assume that $T(\metr(X)) = \metr(Y)$,
 and let $\Phi : \doubl(Y) \to \doubl(X)$ be a bijection as in Proposition~\ref{bij.}.
Define $d(\Phi\{z,w\}) = d(x,y)$ for any $d \in \pseudo(X)$ and any $\{z,w\} \in \doubl(Y)$ with $\Phi(\{z,w\}) = \{x,y\} \in \doubl(X)$.
In this section, we will construct a bijection $\phi : Y \to X$ such that $\Phi(\{x,y\}) = \{\phi(x),\phi(y)\}$ for any $x, y \in Y$ and $\Phi^{-1}(\{z,w\}) = \{\phi^{-1}(z),\phi^{-1}(w)\}$ for any $z, w \in X$,
 and more, such that
 \begin{equation}
  T(d)(x,y) = d(\phi(x),\phi(y)) \tag{$\ast$}
 \end{equation}
 holds for every $d \in \pseudo(X)$.
The map $\Phi$ induces the following equality between $d \in \pseudo(X)$ and $T(d) \in \pseudo(Y)$.

\begin{proposition}\label{eq.}
The equality $T(d)(x,y) = d(\Phi(\{x,y\}))$ holds for every $d \in \pseudo(X)$ and every $\{x,y\} \in \doubl(Y)$.
\end{proposition}

\begin{proof}
First, we prove that $T(d)(x,y) \leq d(\Phi(\{x,y\}))$.
Using Lemma~\ref{transl.}, we can take $\rho \in \mathcal{P}(Y,\{x,y\})$ so that $T(d) + \rho \in \mathcal{P}(Y,\{x,y\}) \cap \metr(Y)$.
By Proposition~\ref{bij.},
 $$T^{-1}(\mathcal{P}(Y,\{x,y\}) \cap \metr(Y)) = \mathcal{P}(X,\Phi(\{x,y\})) \cap \metr(X),$$
 and hence $T^{-1}(T(d) + \rho) \in \mathcal{P}(X,\Phi(\{x,y\}))$.
Then due to Proposition~\ref{norm},
\begin{align*}
 T(d)(x,y) + \rho(x,y) &= (T(d) + \rho)(x,y) = \|T(d) + \rho\| = \|T^{-1}(T(d) + \rho)\|\\
 &= T^{-1}(T(d) + \rho)(\Phi(\{x,y\})).
\end{align*}
Since $T$ is isometry,
 \begin{align*}
  T(d)(x,y) + \rho(x,y) - d(\Phi(\{x,y\})) &= T^{-1}(T(d) + \rho)(\Phi(\{x,y\})) - d(\Phi(\{x,y\}))\\
  &\leq \|T^{-1}(T(d) + \rho) - d\| = \|T^{-1}(T(d) + \rho) - T^{-1}(T(d))\|\\
  &= \|T(d) + \rho - T(d)\| = \|\rho\| = \rho(x,y),
 \end{align*}
 which implies that $T(d)(x,y) \leq d(\Phi(\{x,y\}))$.
Similarly, we get that $T(d)(x,y) \geq d(\Phi(\{x,y\}))$.
The proof is completed.
\end{proof}

Now we shall construct a bijection from $Y$ to $X$ that is compatible with the isometry $T$.

\begin{lemma}\label{singleton}
Assume that the cardinality of $Y$ is greater than $2$.
For each point $y \in Y$, there uniquely exists a point $x \in X$ such that $\{x\} = \bigcap_{z \in Y \setminus \{y\}} \Phi(\{y,z\})$.
\end{lemma}

\begin{proof}
First, fix distinct points $z_1$ and $z_2$ in $Y \setminus \{y\}$,
 so we can uniquely find a point $x \in X$ such that $\{x\} = \Phi(\{y,z_1\}) \cap \Phi(\{y,z_2\})$.
In fact, since $\Phi$ is an injection,
 the cardinality of $\Phi(\{y,z_1\}) \cap \Phi(\{y,z_2\})$ is less than $2$.
Suppose that $\Phi(\{y,z_1\})$ does not meet $\Phi(\{y,z_2\})$.
It follows from Theorem~\ref{ext.} that there is $d \in \pseudo(X)$ such that $d(\Phi(\{y,z_1\})) = d(\Phi(\{y,z_2\})) = 1$ and $d(\Phi(\{z_1,z_2\})) = 3$.
Then by Proposition~\ref{eq.},
\begin{align*}
 3 &= d(\Phi(\{z_1,z_2\})) = T(d)(z_1,z_2) \leq T(d)(y,z_1) + T(d)(y,z_2)\\
 &= d(\Phi(\{y,z_1\})) + d(\Phi(\{y,z_2\})) = 2,
\end{align*}
 which is a contradiction.
Thus the doubletons $\Phi(\{y,z_1\})$ and $\Phi(\{y,z_2\})$ intersect at the only point $x \in X$.
Then we can choose distinct points $w_1, w_2 \in X \setminus \{x\}$ so that $\{x,w_1\} = \Phi(\{y,z_1\})$ and $\{x,w_2\} = \Phi(\{y,z_2\})$.
Suppose that there exists $z \in Y \setminus \{y\}$ such that $\Phi(\{y,z\})$ does not contain the point $x$,
 so by the above argument, $\{w_1,w_2\} = \Phi(\{y,z\})$.
According to Theorem~\ref{ext.}, taking a pseudometric $d \in \pseudo(Y)$ such that $d(y,z_1) = d(y,z_2) = 1$ and $d(y,z) = 3$, we have that by Proposition~\ref{eq.},
\begin{align*}
 3 &= d(y,z) = T^{-1}(d)(\Phi(\{y,z\})) = T^{-1}(d)(w_1,w_2) \leq T^{-1}(d)(x,w_1) + T^{-1}(d)(x,w_2)\\
 &= T^{-1}(d)(\Phi(\{y,z_1\})) + T^{-1}(d)(\Phi(\{y,z_2\})) = d(y,z_1) + d(y,z_2) = 2.
\end{align*}
This is a contradiction.
Hence $\{x\} = \bigcap_{z \in Y \setminus \{y\}} \Phi(\{y,z\})$.
We complete the proof.
\end{proof}

By virtue of the above lemma, we can define a map $\phi : Y \to X$ by $\{\phi(y)\} = \bigcap_{z \in Y \setminus \{y\}} \Phi(\{y,z\})$ for every $y \in Y$.

\begin{proposition}\label{corresp.}
Suppose that the cardinality of $X$ or $Y$ is greater than $2$.
The map $\phi$ is a bijection,
 and $\Phi(\{x,y\}) = \{\phi(x),\phi(y)\}$ for any $x, y \in Y$ and $\Phi^{-1}(\{z,w\}) = \{\phi^{-1}(z),\phi^{-1}(w)\}$ for any $z, w \in X$.
\end{proposition}

\begin{proof}
Recall that the cardinality of $X$ is coincident with the one of $Y$ according to Remark~\ref{card.}.
By the same argument as Lemma~\ref{singleton}, we can define a map $\psi : X \to Y$ by $\{\psi(x)\} = \bigcap_{z \in X \setminus \{x\}} \Phi^{-1}(\{x,z\})$ for every point $x \in X$.
Then $y = \psi(\phi(y))$ for any $y \in Y$.
Indeed, letting $z_1, z_2 \in Y \setminus \{y\}$ be distinct points, we have that $\{\phi(y)\} = \Phi(\{y,z_1\}) \cap \Phi(\{y,z_2\})$ and
 $$\{y\} = \{y,z_1\} \cap \{y,z_2\} = \Phi^{-1}(\Phi(\{y,z_1\})) \cap \Phi^{-1}(\Phi(\{y,z_2\})) = \{\psi(\phi(y))\}.$$
Similarly, $\phi(\psi(x)) = x$ for each $x \in X$.
Therefore $\phi$ is a bijection and $\phi^{-1} = \psi$.
The latter part follows from the definition of $\phi$ and $\psi$.
\end{proof}

\section{Proof of Main Theorem}

Now we shall show Main Theorem.

\begin{proof}[Proof of Main Theorem]
First, the implications (1) $\Rightarrow$ (2), (1) $\Rightarrow$ (3) and (1) $\Rightarrow$ (4) follow from \cite[Lemma~2.4]{IK}.
Indeed, taking any homeomorphism $\psi : Y \to X$, we can define an isometry $S : \pseudo(X) \to \pseudo(Y)$ by $S(d)(x,y) = d(\psi(x),\psi(y))$ for each $d \in \pseudo(X)$ and for any $x, y \in Y$.
Then $S(\adm(X)) = \adm(Y)$ and $S(\metr(X)) = \metr(Y)$.
Second, since $\adm(X)$ and $\metr(X)$ are dense in $\pseudo(X)$ by Proposition~\ref{dense},
 each isometry in (3) and (4) can be extended to the one in (2),
 which means that (3) $\Rightarrow$ (2) and (4) $\Rightarrow$ (2) hold.
Third, we shall show the implication (2) $\Rightarrow$ (1).
In the case where the cardinality of $X$ or $Y$ is less than or equal to $2$, obviously this implication holds, see Remark~\ref{card.}.
In the other case, we need only to prove that the map $\phi : Y \to X$ as in Proposition~\ref{corresp.} is a homeomorphism.
To investigate that $\phi$ is continuous, fix any point $y \in Y$ and any open neighborhood $U$ of $\phi(y)$ in $X$.
According to Theorem~\ref{ext.}, we can find $d \in \pseudo(X)$ such that $d(\phi(y),x) = 2$ if $x \in X \setminus U$,
 and $d(x,x') = 0$ if $x, x' \in X \setminus U$.
Let
 $$V = \{z \in Y \mid T(d)(y,z) < 1\},$$
 so it is an open neighborhood of $y$ in $Y$.
Then for every $z \in V$, $d(\phi(y),\phi(z)) = T(d)(y,z) < 1$,
 and hence $\phi(z) \in U$.
It follows that $\phi$ is continuous.
Similarly, $\phi^{-1}$ is also continuous.
We conclude that $\phi : Y \to X$ is a homeomorphism.
Finally, we shall investigate that the canonical formula ($\ast$) holds.
When the cardinality of $X$ or $Y$ is less than or equal to $2$,
 it is clear.
When the cardinality of $X$ or $Y$ is greater than $2$,
 we will show the uniqueness of the above homeomorphism $\phi$.
Let $\psi : Y \to X$ be a homeomorphism such that for every $d \in \pseudo(X)$ and for any $x, y \in Y$, $T(d)(x,y) = d(\psi(x),\psi(y))$.
Suppose that $\phi(x) \neq \psi(x)$ for some point $x \in Y$.
Fix any $y \in Y \setminus \{x,\phi^{-1}(\psi(x))\}$ and take a pseudometric $\rho \in \pseudo(X)$ such that $\rho(\phi(x),\phi(y)) = 0$ and $\rho(\psi(x),\psi(y)) = 1$, using Theorem~\ref{ext.}.
Then
 $$0 = \rho(\phi(x),\phi(y)) = d(x,y) = \rho(\psi(x),\psi(y)) = 1,$$
 which is a contradiction.
We complete the proof.
\end{proof}

\end{document}